\documentclass{article}

\usepackage[utf8]{inputenc}

\newcommand{\mcC}{\mathcal{C}}
\newcommand{\mcD}{\mathcal{D}}

\newcommand{\calC}{\mathcal{C}}
\newcommand{\calD}{\mathcal{D}}

\newcommand{\calZ}{\mathcal{Z}}

\newcommand{\bbC}{\mathbb{C}}

\newcommand{\mrT}{\mathrm{T}}

\newcommand{\C}{\mathbb{C}}
\newcommand{\Z}{\mathbb{Z}}
\newcommand{\N}{\mathbb{N}}

\newcommand{\Hom}{\operatorname{Hom}}
\newcommand{\Id}{\operatorname{Id}}
\newcommand{\id}{\operatorname{id}}
\newcommand{\Fun}{\operatorname{Fun}}

\newcommand{\Vect}{\operatorname{Vect}}
\newcommand{\Bord}{\mathrm{Bord}}

\newcommand{\Hilb}{\mathrm{Hilb}}
\newcommand{\Hilbfd}{\mathrm{Hilb}^{\rm fd}}
\newcommand{\Vectfd}{\mathrm{Vect}^{\rm fd}}
\newcommand{\Cat}{\mathrm{Cat}}
\newcommand{\dCat}{\dag\mathrm{Cat}}

\newcommand{\ICat}{\mathrm{ICat}}
\newcommand{\PCat}{\mathrm{PCat}}
\newcommand{\indef}{\mathrm{indef}}

\newcommand{\Herm}{\operatorname{Herm}}
\newcommand{\exfix}{{\exists \rm fix}}
\newcommand{\pos}{{\mathrm{pos}}}

\newcommand{\fix}{{\rm fix}}

\usepackage[normalem]{ulem}
\newcommand\scrap{\bgroup\markoverwith
	{\textcolor{red}{\rule[.5ex]{2pt}{0.4pt}}}\ULon}

\newcounter{jfc}

\newcommand{\too}{\longrightarrow}

\newcommand{\pants}{
\parbox[center]{10pt}{\centering
\begin{tikzpicture}[scale=.6]
    \draw (0,.2) ellipse (.05 and .1);
    \draw (0,-.2) ellipse (.05 and .1);
    \draw (0,.3) to [in=180, out=0] (.4,.1);
    \draw (0,-.3) to [in=180, out=0] (.4,-.1);
    \draw (0,.1) to [in=0, out=0] (0,-.1);
    \draw (.4,0) ellipse (.05 and .1);
    \end{tikzpicture}}}
\newcommand{\copants}{
\parbox[center]{10pt}{\centering
\begin{tikzpicture}[scale=.6, xscale=-1]
    \draw (0,.2) ellipse (.05 and .1);
    \draw (0,-.2) ellipse (.05 and .1);
    \draw (0,.3) to [in=180, out=0] (.4,.1);
    \draw (0,-.3) to [in=180, out=0] (.4,-.1);
    \draw (0,.1) to [in=0, out=0] (0,-.1);
    \draw (.4,0) ellipse (.05 and .1);
    \end{tikzpicture}}}
\newcommand{\codisk}{
\parbox[center]{5pt}{\centering
\begin{tikzpicture}[scale=.6]
    \draw (0,0) ellipse (.05 and .1);
    \draw (0,.1) to [in=90, out=0] (.2,0);
    \draw (0,-.1) to [in=-90, out=0] (.2,0);
    \end{tikzpicture}}}
\newcommand{\disk}{
\parbox[center]{5pt}{\centering
\begin{tikzpicture}[scale=.6, xscale=-1]
    \draw (0,0) ellipse (.05 and .1);
    \draw (0,.1) to [in=90, out=0] (.2,0);
    \draw (0,-.1) to [in=-90, out=0] (.2,0);
    \end{tikzpicture}}}
\usepackage{amsmath}
\usepackage{amssymb}
\usepackage{amsthm}

\usepackage{mathtools}

\usepackage{tikz-cd}

\usepackage{hyperref}

\usepackage{todonotes} 

\newtheorem{thm}{Theorem}[section]

\newtheorem{lem}[thm]{Lemma}
\newtheorem{cor}[thm]{Corollary}

\theoremstyle{definition}
\newtheorem{defn}[thm]{Definition}
\newtheorem{obs}[thm]{Observation}
\newtheorem{ex}[thm]{Example}
\newtheorem{rem}[thm]{Remark}

\title{Dagger categories via anti-involutions and positivity}
\author{Luuk Stehouwer and Jan Steinebrunner}

\begin{document}

\maketitle

\begin{abstract}
    
    Dagger categories are an essential tool for categorical descriptions of quantum physics, for example in categorical quantum mechanics and unitary topological field theory.
    Their definition however is in tension with the ``principle of equivalence'' that lies at the heart of category theory, thereby inhibiting generalizations to higher categories.
    In this note we propose an alternative, coherent description of dagger categories based on the well-studied notion of anti-involutions $d: \mathcal{C} \to \mathcal{C}^{op}$, which coherently square to the identity functor $\eta: d^2 \cong \id_{\mathcal{C}}$.
    A general anti-involution need not be the identity on objects, but we instead consider certain isomorphisms $dx \cong x$, which we call Hermitian fixed points as they generalize the notion of a Hermitian inner product on a vector space.
    We define a ``positivity notion" on $(\mathcal{C},d, \eta)$ in terms of such Hermitian fixed points.
    This terminology is motivated by the dagger category of Hilbert spaces, in which case the positivity notion consists of the positive definite pairings.
    Our main result is that the $2$-category of anti-involutive categories with a positivity notion is biequivalent to the $2$-category of dagger categories. 
\end{abstract}

\tableofcontents

\section{Dagger categories}

Hilbert spaces play an important role in the mathematical study of physical systems and in particular in the notion of unitary topological quantum field theory.
In the context of unitary TFTs it is especially important to understand Hilbert spaces from a categorical perspective.

When considering the category of finite dimensional Hilbert spaces and bounded operators $\Hilbfd$, one is faced with a fundamental problem:
the forgetful functor $\Hilbfd \to \Vectfd_\bbC$ is an equivalence of categories.
It is essentially surjective because every finite-dimensional vector space admits a Hilbert space structure and it is fully faithful because every linear map between finite dimensional Hilbert spaces is bounded.
We conclude that in this framework, category theory cannot tell apart Hilbert spaces and vector spaces.
To resolve this, we need to remember how to take the adjoint $A^*\colon H' \to H$ of an operator $A\colon H \to H'$. 
In other words, we should think of $\Hilbfd$ as a dagger category:

A dagger category is a category $\mcC$ equipped with a functor $\dag\colon \mcC^{op} \to \mcC$ satisfying $\dag \circ \dag^{op} = \Id_\mcC$ and $\dag(x) = x$ for all objects $x \in \mcC$.
A dagger functor $F\colon (\calC,\dag) \to (\calD, \ddagger)$ is a functor $F\colon \calC \to \calD$ such that $F(f^\dag) = F(f)^{\ddagger}$ holds for all morphisms $f\colon x \to y$ in $\calC$.

While dagger categories are key to categorical approaches to quantum physics, they also come with an inherent difficulty:
the condition $\dag(x) = x$ behaves poorly under equivalences of categories, and so attempts to transport dagger structures under those will fail in general. 
Dagger categories are hence sometimes humorously referred to as an ``evil'' concept \cite{MO-evil}, 
as they violate this principle of equivalence.
For example, there is no dagger structure on $\Vectfd$ which makes the equivalence $\Hilbfd \to \Vectfd$ into a dagger functor.
Indeed, let $(V,\langle . , . \rangle)$ be a Hilbert space and $(V,2\langle . , . \rangle)$ the same vector space with a scaled inner product.
Then the morphism $\id_V: (V,\langle . , . \rangle) \to (V,2\langle . , . \rangle)$ is not preserved under $\dag$.
However, its image in $\Vectfd$ is the identity on the vector space $V$ and so must be preserved under $\dag$.

However, there is still a well-behaved ``dagger category theory''
obtained by requiring all coherence isomorphisms to be \emph{unitary}.
   A morphism $u\colon x \to y$ in a dagger category $(\calC, \dag)$ is called unitary if $u^\dag\colon x \to y$ is an inverse to $u$,
   i.e.~if $u \circ u^\dag = \id_y$ and $u^\dag \circ u = \id_x$.
   There also is a notion of isometries:
   these are morphisms $i\colon x \to y$ satsifying only $i^\dag \circ i = \id_x$.

We can define a $2$-category $\dCat$ of dagger categories as follows.
Objects are dagger categories, morphisms are dagger functors, and $2$-morphisms are natural transformations $\alpha\colon F \to G$ such that each $\alpha_x \colon Fx \to Gx$ is an isometry.
Requiring that the natural transformations are isometries ensures that all invertible $2$-morphisms are unitary, 
and hence the $2$-category recovers the appropriate notion of equivalence of dagger categories:

\begin{lem}[{\cite[Lemma 5.1]{vicary}}]\label{lem:dagger-equivalence}
    We say that a dagger functor $F\colon \calC \to \calD$ is a \emph{dagger equivalence} if it satisfies the following equivalent conditions:
    \begin{itemize}
        \item 
        $F$ is an equivalence in the $2$-category $\dCat$.
        (i.e.~there is a dagger functor $G\colon \calD \to \calC$ such that $F \circ G$ and $G \circ F$ are unitarily naturally isomorphic to the respective identity functors.)
        \item $F$ is fully faithful and surjective up to unitaries.
        (i.e.~for each $d \in \calD$ there is a $c \in \calC$ such that $F(c)$ is unitarily isomorphic to $d$.)
    \end{itemize}
\end{lem}

 One can make sense of a large collection of categorical notions by replacing `isomorphism' with `unitary isomorphism', such as limits and adjoints \cite{heunen2019limits}.
This also tells us how to transport dagger categories along equivalences:

\begin{thm}[{\cite[Theorem 3.1.3.]{wayofthedagger}}]
    Let $(\mathcal{C},\dag)$ be a dagger category and $F: \mathcal{C} \to \mathcal{D}$ an equivalence in $\Cat$ such that $\id_{\mathcal{C}} \to F^{-1} F$ and $F^{-1} F F^{-1} \to F^{-1}$ are unitary. Then there is a unique dagger structure on $\mathcal{D}$ making $F$ into a dagger equivalence.
\end{thm}

We provide another example of a dagger structure that cannot be transported along an equivalence: since the notion of unitary isomorphism is potentially stricter than isomorphism, the skeleton $\operatorname{sk}\mathcal{D} \hookrightarrow \mathcal{D}$ of a dagger category $\mathcal{D}$ can in general not be made into an equivalence of dagger categories.
   Namely, if $\mathcal{D}$ has two objects that are isomorphic, but not unitarily, then only one of them can be in $\operatorname{sk}\mathcal{D}$,
   and therefore $\operatorname{sk}\mathcal{D} \hookrightarrow \mathcal{D}$ cannot be surjective up to unitaries.
   Instead, the skeleton has to be replaced by a category with one object for each unitary isomorphism class of $\mathcal{D}$. 

The purpose of this note is to compare this $2$-category theory of dagger categories with the $2$-category theory of their coherent analogue: anti-involutive categories.
Additionally, we precisely describe which information is lost in the comparison process.
We define an anti-involutive category to be a category $\calC$ equipped with a functor $d\colon \calC^{op} \to \calC$ that squares to the identity functor up to chosen higher coherence.\footnote{This is sometimes called a category with duality.}
Abstractly, the $2$-category $\ICat$ of anti-involutive categories may be thought of as the homotopy fixed point category of the involution $\calC \mapsto \calC^{op}$ on the $2$-category $\Cat$.
Any dagger category gives rise to an anti-involutive category with trivial higher coherence, and this defines a $2$-functor $\mrT\colon \dCat \to \ICat$.
However, we will see that anti-involutive categories only suffice to capture the behaviour of ``indefinite'' dagger categories (Definition \ref{defn:indefinite}).
\begin{thm}
    There is a $2$-adjunction
    \[
        \mrT \colon \dCat \rightleftarrows \ICat \,:\! \Herm
    \]
    and it restricts to an equivalence between the full $2$-subcategory of indefinite complete dagger categories and the full $2$-subcategory of those anti-involutive categories where each object admits at least one fixed point structure.
\end{thm}

To fully capture dagger categories, we will introduce some extra structure on an anti-involutive category. 
More specifically, we define a Hermitian fixed point $h: x \to dx$ in Definition \ref{defn:hermfixpt} as a homotopy $\Z/2$-fixed point under $\Z/2$-action on the maximal subgroupoid induced by $d$.
The main concept we introduce to reconcile anti-involutive categories with dagger categories is a ``positivity notion" on an anti-involutive category (Definition \ref{defn:positivity}), which is a certain collection of Hermitian fixed points on its objects.
The intuition behind a positivity notion is two-fold:
\begin{enumerate}
    \item is specifies the isomorphisms $dx \cong x$ necessary to make $d$ the identity on objects;
    \item it specifies a collection of Hermitian pairings on the category that we prefer to call positive, compare Example \ref{ex:hilb}. 
\end{enumerate}
We then define $\PCat$ to be the $2$-category of anti-involutive categories equipped with a positivity notion.
This approach is in part motivated by LeFanu Lumsdaine's mathoverflow answer \cite{MO-evil}, which suggests to encode dagger categories by keeping track of a coherent involution and ``unitary fixed point data''.
Our main theorem states that these indeed form an equivalent notion to dagger categories.
\begin{thm}\label{thm:main}
    There the above adjunction lifts to a biequivalence of $2$-categories
    \[
        \dCat \simeq \PCat
    \]
    that commutes with the forgetful functors to $\Cat$.
\end{thm}

There has been plenty of previous work on several notions of involutions on categories, mostly about covariant (sometimes op-monoidal) involutions.
A partial list includes \cite{beggs, jacobs, egger, beninischenkelwoike, henriquespenneys, yau2020involutive, srinivasan2021dagger, srinivasan2023dagger, henriques2023}. 
Even though some of these references relate categories with weak involution to dagger categories, most references work with categories with more structure, such as (symmetric) monoidal or $\C$-linear categories.
We think of our formulation of the relationship between anti-involutive categories and dagger categories as the most elementary relationship, which could be enhanced with more structure if so desired.
In fact, the second author in \cite{luukthesis} obtains a (symmetric) monoidal version of this theorem.

One of the key uses of our main theorem is that it allows us to compute categories of dagger functors from ordinary functor categories together with information about the anti-involutions and the positivity notions.

\begin{thm}
    Let $(\calC, \dag)$ and $(\calD, \dag)$ be two dagger categories.
    Then $F \mapsto \dag_\calD \circ F \circ \dag_\calC$ defines an anti-involution on the category of all (not necessarily dagger) functors $F: \mathcal{C} \to \mathcal{D}$.
    The inclusion of the dagger functors into the fixed points 
    \[
        \Fun^\dag((\calC, \dag), (\calD, \dag)) \hookrightarrow 
        \left(\Fun(\calC, \calD)\right)^{\rm fix}
    \]
    is fully faithful and its essential image consists of those functors that preserve the positivity notions.
\end{thm}

In Section \ref{sec:outlook}, we similarly describe symmetric monoidal dagger functors as certain fixed points on the category of symmetric monoidal functors.
(This uses the aforementioned variant of our main theorem, proved in \cite{luukthesis}.)
Inspired by the approach of \cite{freedhopkins} we use this to study unitary TQFTs, and we will give a concrete example by classifying unitary 2-dimensional TQFTs.

\subsubsection*{Acknowledgements}
The first author would like to thank Theo Johnson-Freyd, Lukas M\"uller, David Reutter, Stephan Stolz, and Peter Teichner for fruitful discussions on dagger categories, especially in relationship to unitary TFTs.
The first author expresses his gratitude to the Max Planck Institute for Mathematics in Bonn for its stimulating research environment leading to this work.

The second author would like to thank Andr\'e Henriques for introducing him to dagger categories, 
and Dave Penneys and David Reutter for many enlightening conversations on this topic.
The second author is supported by the ERC grant no.~772960, and would like to thank the Copenhagen Centre for Geometry and Topology for their hospitality.

We would like to thank Lukas Brantner for pointing out a mistake in the definition of the positivity structure in example \ref{ex:positivity-structure-on-dagger-cat} in the published version of this paper.
This example, remark \ref{rem:positive}, and the relevant parts of the proof of \ref{thm:biequivalence} have been fixed in the present version. 

\section{Anti-involutive categories}

As a first approximation to a more categorically well-behaved version of dagger categories, we can weaken the condition that $\dagger: \mathcal{C} \to \mathcal{C}^{op}$ squares to the identity functor on the nose and we also no longer require it to be the identity on objects.
Instead, we give a natural isomorphism $\dagger \circ \dagger^{op} \cong \Id_\mathcal{C}$ satisfying some compatibility conditions.
Here given a functor $F: \mathcal{C} \to \mathcal{D}$, we denoted the canonical induced functor $\mathcal{C}^{op} \to \mathcal{D}^{op}$ by $F^{op}$, but we will often abuse notation and write it as $F$.

A category with anti-involution is exactly a fixed point for the $\Z/2$-action on the $(2,1)$-category of categories given by $\mathcal{C} \mapsto \mathcal{C}^{op}$, see \cite[section 2.2]{hessethesis} or \cite[Appendix A.2]{mullerstehouwer}.
This results in the following concrete definition:

\begin{defn}
For $\mathcal{C}$ a category, an \emph{anti-involution} is a functor $d: \mathcal{C} \to \mathcal{C}^{op}$ and a natural isomorphism $\eta: \id_{\mathcal{C}} \Rightarrow d^{op} \circ d $ 
such that $\eta_{d(c)}: d(c) \to d d d(c)$ and $d(\eta_c): ddd(c) \to d(c)$ are inverses.
We call the triple $(\mathcal{C}, d, \eta)$ an \emph{anti-involutive category.}
\end{defn}

\begin{rem}
In fact we could require that $d^{op} \circ d = \id_\calC$ and $\eta= \id_{\id_\calC}$.
This would lead to a biequivalent $2$-category, also see \cite[Section 3]{Cockett} where such stricter involutions are studied in the context of linear logic.
However, we will not pursue this strictification here, because many examples are not strict, and it is not in the spirit of this paper.
\end{rem}

\begin{ex}\label{ex:dagger-is-trivial-involutive}
If $(\mathcal{C}, \dagger)$ is a dagger category, define an anti-involution as $d := \dagger$.
Since $dc = c$ for every object $c$ of $\mathcal{C}$, we can set $\eta_c = \id_c$, which gives an anti-involutive category since $\id_c^\dagger = \id_c$.
\end{ex}

\begin{defn}
    An involutive functor $F:(\calC_1,d_1,\eta_1) \to (\calC_2,d_2,\eta_2)$ consists of a functor $F:\calC_1 \to \calC_2$ and a natural isomorphism $\varphi:F^{op} \circ d_1 \cong d_2 \circ F $ such that the following square commutes for all $x \in \calC_1$:
    \[
        \begin{tikzcd}[column sep = 60]
            F(x) \ar[d, "(\eta_2)_{F(x)}"] \ar[r, "F((\eta_1)_x)"] & (F \circ d_1^{op} \circ d_1)(x) \ar[d, "\varphi_{d_1(x)}"]\\
            (d_{2}^{op}\circ d_2 \circ F)(x) \ar[r, "d_2^{op}(\varphi_x)"] & (d_2^{op} \circ F^{op} \circ d_1)(x) 
        \end{tikzcd}
    \]
    The composition of involutive functors $(F: \calC_1 \to \calC_2 ,\varphi) \circ (G: \calC_2 \to \calC_3,\psi)$ is defined to come equipped with the natural transformation
    \[
    F \circ G \circ d_1 (x) \xrightarrow{F(\psi_x)} F \circ d_2 \circ G (x) \xrightarrow{\varphi_{G(x)}} d_3 \circ F \circ G(x)
    \]
    which is easily shown to satisfy the required condition.
    An involutive natural transformation $\alpha\colon (F\colon  \calC_1 \to \calC_2,\varphi) \Rightarrow (G\colon  \calC_1 \to \calC_2,\psi)$ is a natural transformation $\alpha\colon F \Rightarrow G$ such that the following square commutes for all $x \in \calC_1$:
    \[
        \begin{tikzcd}[column sep = 40]
            (F \circ d_1)(x) \ar[d, "\varphi_x"] \ar[r, "\alpha_{d_1(x)}"] & (G \circ d_1)(x) \ar[d, "\psi_x"]\\
            (d_2 \circ F)(x) & (d_2 \circ G)(x) \ar[l, "d_2(\alpha_x)"] 
        \end{tikzcd}
    \]
    The composition of involutive natural transformations is involutive.
    Let $\ICat$ denote the $2$-category of anti-involutive categories, involutive functors, and involutive natural transformations.
\end{defn}

Note that similar to isometric natural transformations for dagger categories, an involutive natural transformation $\alpha_c$ admits a left inverse, but not necessarily a right inverse.

\begin{rem}
The observation in Example \ref{ex:dagger-is-trivial-involutive} that every dagger category is canonically an anti-involutive category extends to give a $2$-functor $\mrT\colon \dCat \to \ICat$.
More precisely, if $F\colon (\mathcal{C}, \dagger) \to (\mathcal{D}, \dagger)$ is a dagger functor, we can take $\varphi_c = \id_{F(c)}$.
The condition that this defines a natural transformation $F \circ \dagger \Rightarrow \dagger \circ F$ is equivalent to $F$ being a dagger functor.
The remaining condition between $\eta$ and $\varphi$ is satisfied, since all morphisms involved are the identity.

Finally, let $\alpha\colon F \Rightarrow F'$ be a natural transformation.
Then $\alpha$ is an involutive natural transformation between the induced involutive functors if and only if $\alpha_c$ is an isometry for all objects $c$, which is how we defined $2$-morphisms in $\dCat$.
We recall that that $\alpha_c$ need not be invertible, but it is invertible if and only if it is unitary.
Clearly these constructions preserve composition of functors and both horizontal and vertical composition of natural transformations.
\end{rem}

\begin{lem}\label{lem:involutive-equivalence}
    An involutive functor $(F, \varphi)\colon (\mcC, d, \rho) \to (\mcD, d, \eta)$ is an equivalence in $\ICat$ (i.e.~it has an involutive inverse up to \emph{involutive} natural transformation)
    if and only if the underlying functor $F\colon\mcC \to \mcD$ is an equivalence of categories.
\end{lem}
\begin{proof}
    The only if direction holds because if $(G, \psi)$ is an involutive inverse functor, then $G$ is an inverse of $F$ up to natural isomorphism.
    
    For the if direction, pick some $G\colon\mcD \to \mcC$ and natural transformations $\alpha\colon F \circ G \cong \Id_\mcD$ and $\beta\colon G \circ F \cong \Id_\mcC$.
    Recall that without loss of generality, we can assume this is an adjoint equivalence, i.e.~$\alpha$ and $\beta$ satisfy the snake identities.
    It suffices to provide the data $\psi$ that makes $G$ into an involutive functor and show that $\alpha$ and $\beta$ become involutive natural transformations.
    Define $\psi$ at an object $y$ of $\mcD$ as
    \[
    Gdy \xrightarrow{Gd \alpha_y} GdFGy \xrightarrow{G \varphi_{Gy}^{-1}} GFdGy \xrightarrow{\beta_{dGy}} dGy.
    \]
    By definition of being an involutive functor, we have to show the diagram
    \[
    \begin{tikzcd}[column sep=1.4cm]
        Gy \ar[r,"{G\rho_y}"] \ar[d,"{\eta_{Gy}}"] & G d^2 y \ar[r,"{G d \alpha_{dy}}"] & GdFG dy & GFdGdy \ar[l,"{G \varphi_{G dy}}"] \ar[d,"{\beta_{dGdy}}"]
        \\
        d^2 Gy \ar[r,"{d \beta_{d Gy}}"] & dGF dGy & dG dFGy \ar[l,"{dG\varphi_{Gy}}"] \ar[r,"{dGd\alpha_y}"] & dGdy
    \end{tikzcd}
    \]
    commutes.
    For this, first note that 
    \[
    \begin{tikzcd}[column sep=1.7cm]
        \ & Gy \ar[r,"{\eta_{Gy}}"] & d^2 Gy
        \\
        Gy \ar[ru,equals] \ar[r,"{G \alpha_y}"] \ar[d,"{G \rho_y}"] & GFG y \ar[u,"{\beta_{Gy}}"] \ar[r,"{GF \eta_{Gy}}"] \ar[d,"{G \rho_{FG y}}"] & GF d^2 G y  \ar[u,"{\beta_{d^2 G y}}"] \ar[d,"{G \varphi_{dGy}}"]
        \\
        G d^2 y & G d^2 FG y \ar[r,"{Gd\varphi_{Gy}}"] \ar[l,"{G \alpha_{d^2 y}}"] & GdFdGy
    \end{tikzcd}
    \]
    commutes.
    Indeed, the left upper triangle commutes by the snake identity, the right upper square commutes by naturality of $\beta$, the left lower square commutes by naturality of $\alpha$ and the lower right square commutes because $F$ is an involutive functor.
    Replacing the morphisms $G\rho_y$ and $\eta_{Gy}$ in the first diagram by this second diagram leads us to conclude that it suffices to show that the following diagram commutes. 
    We omitted the choice of input object $y$ in $\mcD$ from the notation for reasons of space.
    \[
    \begin{tikzcd}[row sep=1.35cm]
        \ & G d^2 FG \ar[d,"{Gd \alpha_{dFG}}"] \ar[ld,"{G \alpha_{d^2}}", swap] \ar[r,"{G d \varphi_G}"] & 
        GdFdG \ar[d,"{GdF\beta_{dG}}", bend left] \ar[d,"{Gd\alpha_{FdG}}", bend right, swap] & 
        GFd^2 G \ar[d,"{GFd\beta_{dG}}", swap] \ar[dr,"{\beta_{d^2 G}}"] \ar[l,"{G \varphi_{dG}}", swap] & \ 
        \\
        G d^2 \ar[dr,"{G d \alpha_d}", swap] & GdFG d FG \ar[r,"{GdFG \varphi_G}", swap] \ar[d,"{GdFGd \alpha}"] & G d FGF d G & GFdGFdG \ar[l,"{G \varphi_{GFdG}}"]  \ar[d,"{\beta_{dGFdG}}"]& d^2 G \ar[dl,"{d \beta_{dG}}"]
        \\
        \ & G dFG d & GFdGdFG \ar[ul,"{G\varphi_{GdFG}}", swap] \ar[ur,"{GFdG\varphi_G}"] \ar[dl,"{GFdGd\alpha}"] \ar[dr,"{\beta_{dGdFG}}", swap] & dGFdG & \ 
        \\
        \ & GFdGd \ar[u,"{G \varphi_{Gd}}"] \ar[r,"{\beta_{dGd}}"] & dGd & dGdFG\ar[u,"{dG \varphi_G}"] \ar[l,"{dGd\alpha}"] & \
    \end{tikzcd}
    \]
    Every quadrilateral in the diagram commutes by the interchange law and the upper two bent arrows are equal by the snake identity.
    We are led to conclude that $(G, \psi)$ is an involutive functor.
    
    It remains to show that $\alpha$ and $\beta$ are involutive natural transformations.
    Writing out the definition of the involutive structure on $F \circ G$ this entails that for $\alpha$ we have to show that the diagram 
    \[
    \begin{tikzcd}[column sep=1.2cm]
        FGdFG y \ar[ddr,"{\alpha_{dFGy}}"] & FG dy \ar[l,"{FG d\alpha_y}", swap] \ar[d,"{\alpha_{dy}}"]
        \\
        FGFdGy \ar[u,"{FG \varphi_{Gy}}"] \ar[d,"{\alpha_{FdGy}}", bend left] \ar[d,"{F \beta_{d G y}}", bend right, swap] & dy \ar[d,"{d \alpha_y}"]
        \\
        FdGy \ar[r,"{\varphi_{Gy}}"] & dFG y
    \end{tikzcd}
    \]
    commutes.
    The two bent arrows are equal by the snake identity and the other two parts commute by the interchange law.
    The proof that $\beta$ is involutive is analoguous.
\end{proof}

The following example shows that the underlying anti-involution of a dagger category does not preserve enough information.

\begin{ex}
\label{ex:Hilbpmequiv}
    Let $\operatorname{Herm}_{\C}$ denote the category where objects are finite dimensional complex vector spaces with a non-degenerate sesquilinear form such that 
    \[
    \langle v,w \rangle = \overline{\langle w,v \rangle}
    \]
    and morphisms are all linear maps.
    In other words, these are Hermitian vector spaces that are not necessarily positive definite inner product spaces.
    It becomes a dagger category when the dagger is defined by taking the adjoint with respect to the pairing.
    
    The category of finite dimensional Hilbert spaces is a full dagger subcategory
    $\Hilb \subset \operatorname{Herm}_{\C}$
    characterised by the condition that the sesquilinar form be positive definite.
    This inclusion is not a dagger equivalence, as it is not surjective up to unitaries.
    Indeed, objects in $\operatorname{Herm}_{\C}$ are classified, up to unitary isomorphism, by their signature $(p,q)$ and the full subcategory only contains those of signature $(p,0)$.
    
However, the inclusion $F\colon \Hilb \to \operatorname{Herm}_{\C}$ is an equivalence of anti-involutive categories.
It is fully faithful and essentially surjective because every finite-dimensional vector space admits some Hilbert space structure.
By Lemma \ref{lem:involutive-equivalence}, this is an equivalence of anti-involutive categories.

Concretely, we could construct a (highly noncanonical) inverse of this equivalence as follows.
Pick for every finite-dimensional Hermitian vector space $(V, \langle .,. \rangle )$ a basis $\alpha_V\colon V \cong \C^n$ once and for all.
Define the functor $G\colon \operatorname{Herm}_{\C} \to \Hilb$ on objects by $G(V, \langle .,. \rangle) = (\C^n, \langle .,. \rangle_{st})$ where $\langle .,. \rangle_{st}$ is the standard Hilbert space structure.
On morphisms we set $G(f\colon V_1 \to V_2) := \alpha_{V_2}^{-1} \circ f \circ \alpha_{V_1}$.
There is an associated canonical natural isomorphism $\alpha\colon \id_{\operatorname{Herm}_{\C}} \implies F \circ G$ given by $\alpha(V, \langle .,. \rangle) = \alpha_V\colon (V, \langle .,. \rangle) \to (\C^n, \langle .,. \rangle_{st})$.
Now, $G$ is not a dagger functor since $\alpha_V$ is in general not unitary.
But even though the anti-involutions $d$ on both categories are the identity on objects, we can use the recipe in the above lemma to equip $G$ with a non-trivial structure of an involutive functor:
\[
\varphi_V\colon G(dV) = G(V) = \C^n \xrightarrow{\alpha_V^\dagger} V \xrightarrow{\alpha_V} \C^n = G(V) = d G(V)
\]
Then the condition that $\varphi$ has to satisfy for $G$ to be an involutive functor boils down to $\varphi^\dagger_V = \varphi_V$, which is easy to check.
Hence $(G,\varphi)$ is an involutive inverse of the involutive functor $F$.
\end{ex}

\section{Hermitian fixed points and Hermitian completion}

In the last section, we proposed the notion of an anti-involutive category as a better-behaved analogue of the notion of a dagger category so that every dagger category has an underlying anti-involutive category.
However, in example \ref{ex:Hilbpmequiv} we found that there are important examples of dagger categories that are equivalent as anti-involutive categories but not as dagger categories.
Heuristically, the example gives us the idea that the dagger category of finite-dimensional Hilbert spaces is not equivalent to the dagger category of finite-dimensional Hermitian vector spaces because in the former fewer Hermitian structures are allowed.
Therefore we study an abstraction of the notion of a Hermitian structure, which we learned from \cite[Definition B.14]{freedhopkins}.

\begin{defn}
\label{defn:hermfixpt}
A \emph{Hermitian fixed point} in a category $\mathcal{C}$ with anti-involution $(d,\eta)$ on an object $c$ is an isomorphism $h\colon c \to dc$ such that
\[
\begin{tikzcd}
c  \arrow[rr, bend right, "h"] \arrow[r, "\eta_c"] &  d^2 c \arrow[r, "dh", swap] & dc
\end{tikzcd}
\]
commutes.
The \emph{adjoint} $f^\dagger\colon c_2 \to c_1$ of a morphism $f\colon c_1 \to c_2$ with respect to Hermitian fixed points $h_1\colon c_1 \to dc_1$ and $h_2\colon c_2 \to dc_2$ is the composition
\[
c_2 \xrightarrow{h_2} dc_2 \xrightarrow{d f} dc_1 \xrightarrow{h_1^{-1}} c_1.
\]
\end{defn}

\begin{ex}
\label{ex:Hermvect}
Take $\mathcal{C} = \Vectfd_\C$ to be the category of finite-dimensional vector spaces.
Recall that the complex conjugate $\overline{V}$ of a vector space $V$ is defined to be the same abelian group but with complex conjugate scalar multiplication.
This extends to a functor $\overline{(.)}: \Vectfd_\C \to \Vectfd_\C$.
Set $d = \overline{(.)}^*$ so that there is an obvious $\eta$ given by the evaluation map. 
It is straightforward to check that $\eta$ satisfies $\eta_{\overline{V}^*} = \overline{\eta_V}^*$.
A Hermitian fixed point consists of a vector space $V$ and an isomorphism $V \to \overline{V}^*$ satisfying a condition.
Such an isomorphism is equivalent to a nondegenerate sesquilinear pairing and the condition is equivalent to the Hermiticity axiom
\[
\langle v,w \rangle = \overline{ \langle w,v \rangle}.
\]
The adjoint is given by the usual adjoint of a linear map.
\end{ex}

Hermitian fixed points naturally form a category $\mathcal{C}^{\fix}$ in which morphisms $f\colon (c_1, h_1) \to (c_2, h_2)$ are morphisms $c_1 \to c_2$ satisfying the compatibility relation 
\[
\begin{tikzcd}
    c_1 \ar[r,"f"] \ar[d,"h_1"] & c_2 \ar[d,"h_1"]
    \\
    dc_1 & dc_2 \ar[l,"df"]
\end{tikzcd}
\]
Let $f\colon (c_1, h_1) \to (c_2,h_2)$ be a morphism in $\mathcal{C}^{\fix}$.
Note that the condition $f$ has to satisfy exactly says that $f^\dagger$ is a left inverse of $f$. 
Therefore $\calC^{\fix}$ is exactly the wide subcategory of isometries of the dagger category $\Herm(\calC)$ that we shall define now.
The construction is closely related to the `unitary core of a $\dagger$-isomix category', which appears in the context of $\dagger$-linear logic \cite[Definition 5.12]{srinivasan2021dagger}.
One could think of $\Herm(\calC)$ as the `co-free dagger category on an anti-involutive category'. This idea is made precise by the $2$-adjunction that we will establish in theorem \ref{thm:2-adjunction}.

\begin{defn}\label{defn:Herm}
The \emph{Hermitian completion} $\operatorname{Herm} \mathcal{C}$ of the anti-involutive category $(\mathcal{C}, d, \eta)$ is the category in which objects are Hermitian fixed points $(c,h)$ and morphisms $(c_1, h_1) \to (c_2,h_2)$ are simply given by morphisms $f\colon c_1 \to c_2$.
\end{defn}

\begin{lem}\label{prop:dagger-on-Herm}
The adjoint on the category $\Herm \mathcal{C}$ makes it into a dagger category.
\end{lem}
\begin{proof}
Let $(c,h)$ be an object of $\mathcal{C}$ with Hermitian fixed point $h\colon c \to dc$.
We have that 
\[
\id_c^\dagger = h^{-1}\circ  d (\id_c) \circ h = h^{-1} \circ \id_{dc} \circ h = \id_c.
\]
If $f\colon (c_1,h_1) \to (c_2,h_2)$ and $g\colon (c_2,h_2) \to (c_3,h_3)$, then $(g \circ f)^\dagger = f^\dagger \circ g^\dagger$ follows from the fact that
\[
c_3 \overset{h_{3}}{\cong} dc_3 \xrightarrow{dg} dc_2 \overset{h_{2}^{-1}}{\cong} c_2 \overset{h_{2}}{\cong} dc_2 \xrightarrow{df} dc_1 \overset{h_{1}}{\cong} c_1
\]
is equal to 
\[
c_3 \overset{h_{3}}{\cong} dc_3 \xrightarrow{d(g \circ f)}  dc_1 \overset{h_{1}}{\cong} c_1
\]
by functoriality of $d$.
Now $f^{\dagger \dagger}$ is the composition
\[
c_1 \overset{h_{1}}{\cong} dc_1 \overset{dh^{-1}_{1}}{\cong} d^2 c_1 \xrightarrow{d^2 f} d^2 c_2 \overset{dh_2}{\cong} dc_2 \overset{h_2^{-1}}{\cong} c_2.
\]
Using the fixed point property of a Hermitian structure, this composition is equal to
\[
c_1 \overset{\eta_{c_1}}{\cong} d^2 c_1 \xrightarrow{d^2 f} d^2 c_2 \overset{\eta_{c_2}^{-1}}{\cong} c_2.
\]
By naturality of $\eta$ this composition is equal to $f$.
\end{proof}

\begin{ex}
The Hermitian completion of $(\mathcal{C} = \Vectfd_\C,d = \overline{(.)}^*)$ is the dagger category of Hermitian vector spaces we considered in example \ref{ex:Hilbpmequiv}.
So $\Herm \Vectfd_\C = \Herm_\C$
\end{ex}

\begin{rem}
Unlike for finite-dimensional vector spaces, $\overline{(.)}^*$ does not define an anti-involution on infinite-dimensional vector spaces.
Indeed, even though there is still a well-defined bidual map $\eta\colon V \to \overline{\overline{V}^*}^*$, it is only injective but not surjective.
Hence the dagger category of all Hilbert spaces can not be constructed in a similar fashion as the last example. 
It would be interesting to study a weakened version of anti-involutive categories in which $\eta$ is not necessarily an isomorphism and Hermitian fixed points $h\colon c \to dc$ are not necessarily isomorphisms.
The technical disadvantage of such a theory would be that we might have to restrict the morphisms in the Hermitian completion to those that admit an adjoint, for example the bounded operators for Hilbert spaces.
An alternative approach would be to work with a certain category of topological vector spaces and use a continuous linear dual.
\end{rem}

\begin{ex}
Let $\mrT \mathcal{C} \in \ICat$ be a dagger category $\mathcal{C}$ seen as an anti-involutive category.
The Hermitian completion $\Herm (\mrT \mathcal{C})$ concretely consists of pairs $(c,\tau)$, where $\tau\colon c \to c$ is invertible and self-adjoint.
For $f\colon (c_1, \tau_1) \to (c_2, \tau_2)$ the new adjoint $*$ on the Hermitian completion is defined as $f^* = \tau_2 \circ f^\dagger \circ \tau_1^{-1}$.
For example, starting with the dagger category of Hilbert spaces, new objects are triples $(V,(-,-),A)$ consisting of a Hilbert space and a self-adjoint invertible linear operator on $V$.
The adjoints of morphisms between such objects are defined using the Hermitian pairing $(-,A-)$ on $V$.
The resulting dagger category is unitarily equivalent to the dagger category of Hermitian vector spaces.
\end{ex}

\begin{ex}
    Again take $\mathcal{C} = \Vectfd_\C$ to be the category of finite-dimensional vector spaces.
    Now define $d$ to be the dual $(-)^*\colon \mathcal{C} \to \mathcal{C}^{op}$ and $\eta\colon V \to V^{**}$ the evaluation map.
    Then a Hermitian fixed point on a vector space $V$ is equivalent to a nondegenerate symmetric bilinear form on $V$.
    More generally, we could take $\mathcal{C}$ to be finite-dimensional complex representations of a finite group $G$.
    Since for a general $G$-representation $V$, there is no $G$-equivariant isomorphism $V \cong V^*$ there are representations that do not admit the structure of a Hermitian fixed point at all.
\end{ex}

\begin{lem}
    \label{functorsfix}
    Let $(\mathcal{C},d_\calC, \eta_\calC), (\calD, d_\calD, \eta_\calD)$ be two anti-involutive categories. 
    Then there is an anti-involutive structure on the category $\Fun(\calC, \calD)$ of functors between them, 
    such that the category of $\Z/2$-fixed points $\Fun(\mathcal{C}, \mathcal{D})^{\fix}$ is the category $\Hom_{\ICat}(\mathcal{C}, \mathcal{D})$ of $1$-morphisms in $\ICat$.
\end{lem}

\begin{proof}
    The functor category $\Fun(\calC, \calD)$ becomes an anti-involutive category via
    \[
        d F := d_\calD \circ F \circ d_\calC.
    \]
    Namely, we can define the anti-involution on natural transformations $\alpha\colon F_1 \Rightarrow F_2$ between functors $F_1, F_2\colon \calC \to \calD$ as the whiskering
    \[
    d \alpha := \id_{d_\calD} \bullet \alpha \bullet \id_{d_\calC},
    \]
    where we denoted horizontal composition of natural transformations with $\bullet$.
    This defines a functor $\Fun(\calC, \calD) \to \Fun(\calC, \calD)$. 
    Define the natural transformation $\eta\colon \id_{\Fun(\calC, \calD)} \Rightarrow d^2$ on $F \in {\Fun(\calC, \calD)}$ by
    \begin{align*}
        F \xrightarrow{\eta_\calD \bullet \id_F \bullet  \eta_\calC} d^2_\calD F d^2_\calC,
    \end{align*}
    which is natural by the exchange law.
    Finally, we have to show that $\eta_{d F} = d \eta_F^{-1}$ for all $F \in \Fun(\calC, \calD)$.
    This amounts to showing that $\eta_\calD \bullet \id_{d_\calD F d_\calC} \bullet \eta_\calC$ is the inverse of $\id_{d_\calC} \bullet \eta_\calD \bullet \id_F \bullet \eta_\mathcal{C} \bullet \id_{d_\calD}$.
    This holds because, since $\eta_\calD$ is part of an anti-involution it satisfies that $\eta_\calD \bullet \id_{d_\calD}$ is inverse to $\id_{d_\calC} \bullet \eta_\calD$ and similarly for $\eta_\calC$.

    A Hermitian fixed point in $\Fun(\calC, \calD)$ is equivalent to an involutive functor.
    Indeed, let $\psi\colon F \Rightarrow d_\calD F d_\calC$ be a Hermitian fixed point on $F$.
    Writing out the condition results in the commutative diagram
    \[
    \begin{tikzcd}[column sep = 50, row sep = 25]
        F c \ar[r,"\psi_c"] \ar[d,"{F ((\eta_\calC)_c)}", swap] & d_\calD F d_\calC c  
        \\
        F d_\calC^2 c \ar[r,"{(\eta_\calD)_{F d_\calC^2 c}}"] & d_\calD^2 F d_\calC^2 c \ar[u,"{d_\calD \psi_{d_\calC c}}", swap]
    \end{tikzcd}
    \]
    for every object $c$.
    A diagram chase shows that under mapping $\psi$ to the composition $\varphi$ defined by
    \[
    \varphi\colon Fd_\calC \xRightarrow{\psi \bullet \id_{d_\calD}} d_\calD F d^2_\calC \xRightarrow{\id_{d_\calD F} \bullet \eta_\calC^{-1}} d_\calD F
    \]
    this becomes the condition that $(F,\varphi)$ is an involutive functor. 
    So we see that $\Herm \Fun(\mathcal{C}, \mathcal{D})$ is the category with objects involutive functors and as morphisms all natural transformations.
    A natural transformation is involutive if and only if it is an isometry in $\Herm \Fun(\mathcal{C}, \mathcal{D})$.
\end{proof}

\begin{ex}
\label{ex:dualhermstructure}
    Let $h\colon c \to dc$ be a Hermitian fixed point.
    Then $(dh)^{-1}\colon dc \to d^2 c$ is a Hermitian fixed point structure on $d c$.
    Indeed, taking $d$ of the diagram saying that $h$ is a fixed point and using that $d \eta_c = \eta_{dc}^{-1}$ yields
    \[
    \begin{tikzcd}
        dc \ar[r,"\eta_{dc}"] & d^3 c  & \ar[l,"d^2 h", swap] d^2 c \ar[ll, "dh", bend left]
    \end{tikzcd}
    \]
    Using that $(d^2 h)^{-1} = d(dh^{-1})$, this diagram indeed expresses the fact that $(dh)^{-1}\colon dc \to d^2 c$ is a Hermitian fixed point.
    Note that by construction $h\colon c \to dc$ is a unitary isomorphism between the objects $(c,h)$ and $(dc, (dh)^{-1})$ in the dagger category $\Herm \mathcal{C}$.
\end{ex}

\begin{rem}
We expect the discussion above to be closely related to \cite[Section 6]{egger} as follows.    
This reference considers covariant op-monoidal involutions which in certain rigid monoidal categories should be related to monoidal anti-involutions after composing with a choice of dual functor.
This relationship is shown in \cite[Section 2.2]{luukthesis} in the symmetric monoidal case.
The notion of Hermitian sesquilinear pairing in \cite[Definition 6.2]{egger} should be related to our notion of Hermitian fixed point and \cite[Lemma 6.3]{egger} should be related to our Hermitian completion.
\end{rem}

\begin{defn}
    We extend the construction of definition \ref{defn:Herm} to a $2$-functor
    \[
        \Herm \colon \ICat \longrightarrow \dCat
    \]
    as follows.
    For an involutive functor $(F, \varphi)\colon (\mcC, d, \eta) \to (\mcD, d, \rho)$ we define 
    \[
    \operatorname{Herm} F\colon \Herm \mathcal{C} \to \Herm \mathcal{D}
    \]
    on objects by 
    $\Herm F(c,h) = (F(c), h_F := \varphi_c \circ F(h))$,
    and on morphisms by $\Herm F(f) = F(f)$.
    For an involutive natural transformation 
    $\alpha \colon (F,\varphi) \to (F',\varphi')$ 
    we define $\Herm \alpha \colon \Herm F \to \Herm F'$ by $(\Herm \alpha)_c := \alpha_c$.
\end{defn}

\begin{lem}
    The above yields a well-defined $2$-functor.
\end{lem}
\begin{proof}
    We have already checked that $\Herm(\calC, d, \eta)$ is indeed a dagger category,
    so next we need to verify that $\Herm(F,\varphi)$ is a dagger functor.
    First, note that $h_F := \varphi_c \circ F(h)\colon F(c) \to F(dc) \cong dF(c)$ is indeed a Hermitian fixed point because of the diagram:
    \[\begin{tikzcd}
	{F(c)} & {F(dc)} & {dF(c)} \\
	& {F(ddc)} \\
	{ddF(c)} & {dF(dc)} & {dF(c)}
	\arrow["{F(h)}", from=1-1, to=1-2]
	\arrow["{\varphi_c}", from=1-2, to=1-3]
	\arrow["{F(\eta_c)}"', from=1-1, to=2-2]
	\arrow["{\eta_{F(c)}}"', from=1-1, to=3-1]
	\arrow["{F(dh)}"', from=2-2, to=1-2]
	\arrow["{d\varphi_c}", from=3-1, to=3-2]
	\arrow["{dF(h)}", from=3-2, to=3-3]
	\arrow["{\varphi_{dc}}"{description}, from=2-2, to=3-2]
	\arrow[Rightarrow, no head, from=1-3, to=3-3]
	\arrow["{h_F}", bend left = 25, from=1-1, to=1-3]
	\arrow["{d(h_F)}"{description}, bend right = 25, from=3-1, to=3-3]
    \end{tikzcd}\]
    Here the triangle commutes because $h$ is a Hermitian fixed point,
    the trapezoid commutes because $\varphi$ is part of an involutive functor,
    and the rectangle commutes because $\varphi$ is a natural transformation.
    
    $\Herm F$ is certainly functorial seeing as morphisms in $\Herm \calC$ are simply composed by composing them in $\calC$.
    To show it is a dagger functor, let 
    \[
    f\colon (c_1, h) \to (c_2, h')
    \]
    be a morphism in $\Herm \mathcal{C}$.    
    Then    
    \[ 
    \Herm F(f^\dagger) = F( c_2 \xrightarrow{h'} dc_2 \xrightarrow{df} dc_1 \xrightarrow{h^{-1}} c_1) 
    = F(h)^{-1} F(df) F(h')    
    \] 
    Recall that since $F$ is involutive, the diagram    
    \[   
    \begin{tikzcd} 
    F(dc_2)  \arrow[d, "\varphi_{c_2}", swap] \arrow[r, "F(df)"] & F(dc_1) \arrow[d, "\varphi_{c_1}"] 
    \\   
    dF(c_2) \arrow[r, "dF(f)"] & dF(c_1)  
    \end{tikzcd}  
    \]    
    commutes.
    Looking at the definition of $h_F , h'_F$, we obtain 
    \[
    \Herm F(f^\dagger) = \Herm F(f)^\dagger.
    \]
    To conclude that $\Herm$ is a $1$-functor we need to check that for two composable involutive functors $(F_1, \varphi_1)$ and $(F_2, \varphi_2)$ we have that 
    $\Herm(F_2) \circ \Herm(F_1) = \Herm(F_2 \circ F_1)$.
    It will suffice to check that both sides do the same on an object $(c, h)$.
    The two resulting Hermitian fixed points on $F_2F_1(c)$ are
    \[\begin{tikzcd}
	{F_2F_1(c)} & {F_2F_1(dc)} && {dF_2F_1(c)} \\
	{F_2F_1(c)} & {F_2F_1(dc)} & {F_2dF_1(c)} & {dF_2F_1(c)}
	\arrow["{F_2F_1(h)}", from=1-1, to=1-2]
	\arrow["{(\varphi_{12})_c}", from=1-2, to=1-4]
	\arrow["{F_2F_1(h)}", from=2-1, to=2-2]
	\arrow[equal, from=1-1, to=2-1]
	\arrow[equal, from=1-2, to=2-2]
	\arrow["{(\varphi_1)_c}", from=2-2, to=2-3]
	\arrow["{(\varphi_2)_{F_1c}}", from=2-3, to=2-4]
	\arrow[Rightarrow, no head, from=1-4, to=2-4]
    \end{tikzcd}\]
    These are indeed the same: the right rectangle commutes because of how the coherence isomorphism $\varphi_{12}$ of the composite functor is defined.
    
    Finally, we need to consider the effect of $\Herm$ on $2$-morphisms.
    Here all there is to check that $\Herm(\alpha)$ is indeed an isometry. 
    This follows from the diagram: 
    \[
    \begin{tikzcd}[column sep = 18]
        F(c) \ar[r,"\alpha_c"] \ar[d,"F(h)", swap] & F'(c) \ar[d,"F'(h)"]
        \\
        F(dc) \ar[d,"\varphi_c", swap] \ar[r,"\alpha_{dc}"] & F'(dc)  \ar[d,"\varphi'_c"]
        \\
        dF(c) & dF'(c) \ar[l,"{d \alpha_c}",swap]
    \end{tikzcd}
    \]
    The squares commute because $\alpha$ is a natural transformation and because $\alpha$ is involutive with respect to $(F,\varphi)$ and $(F', \varphi')$.
    The vertical composites are the Hermitian fixed points $h_F$ and $h_F'$,
    and therefore the diagram shows that $\alpha_c$ is a one-sided inverse to $\alpha_c^\dagger = h_F^{-1} \circ d(\alpha_c) \circ h_F$.
\end{proof}

\section{Indefinite dagger categories}

The $2$-functors $\Herm$ and $\mrT$ are not inverses of each other for two reasons:
\begin{enumerate}
    \item The anti-involutive category $\mrT(\calC, \dag)$ has the property that every object admits at least one Hermitian fixed point structure.
    This is not true for every anti-involutive category, for instance the discrete category $\mathbb{Z}/2$ with the non-trivial swap, and therefore $\mrT$ is not surjective up to equivalence.
    \item There exist dagger categories that are not unitarily equivalent, but become equivalent as anti-involutive categories after applying $\mrT$.
\end{enumerate}

However, we will still be able to show that $\Herm$ and $T$ restrict to a biequvialence between certain full $2$-subcategories.
On the side of the anti-involutive categories we make the following restriction, motivated by point 1 above:

\begin{defn}
    Let $\mathcal{C}^{\exfix}$ denote the full subcategory of the anti-involutive category $\mathcal{C}$ on the objects $c$ that admit some Hermitian fixed point $h\colon c \to dc$.
    This is again an anti-involutive category, also see example \ref{ex:dualhermstructure}.
    Let $\ICat^{\exfix} \subset \ICat$ denote the full $2$-subcategory on those anti-involutive categories in which every object admits some Hermitian fixed point structure.
\end{defn}

To find the correct property on the dagger category side, we note:
\begin{ex}
Consider the dagger category $\Herm_\C$ as a category with anti-involution.
Its Hermitian completion is again dagger-equivalent to $\Herm_\C$.
However, for $\Hilb$ it is instead $\Herm_\C$ which is not dagger-equivalent to $\Hilb$.
Recall that in $\Hilb$ an operator $T: V \to V$ is called positive definite if it is of the form $T = A^\dag A$ for some isomorphism $A: V \to W$.
Note that in $\Hilb$ not every self-adjoint automorphism is positive definite.
However, in $\Herm_\C$ it turns out that every self-adjoint automorphism $T: V \to V$ can be written as $T = A^\dag A$ for some isomorphism $A: V \to W$ to a suitable (possibly mixed signature) Hermitian vector space.
 This is the essential property that $\Herm_\C$ has and $\Hilb$ lacks.
\end{ex}

Motivated by the above example, we want to single out dagger categories in which `every self-adjoint automorphism is positive definite', compare Remark \ref{rem:positive}. 
In analogy with $\Herm_\C$, we think of such dagger categories as containing `all Hermitian forms, even all the indefinite ones'. 

\begin{defn}
\label{defn:indefinite}
    We say that a dagger category $\calD$ is \emph{indefinite} if 
    for any object $x \in \calD$ and any self-adjoint automorphism $a = a^\dag\colon x \cong x$ there is another object $y \in \calD$ and an isomorphism $f\colon x \cong y$ such that $a = f^\dag \circ f$.
    We let $\dCat^\indef \subset \dCat$ denote the full sub-$2$-category on the indefinite complete dagger categories.
\end{defn}

\begin{lem}
    For any anti-involutive category $(\calC,d,\eta)$ the dagger category $\Herm(\calC)$ is indefinite.
\end{lem}
\begin{proof}
    A self-adjoint automorphism is an isomorphism $a\colon (c,h) \to (c,h)$ such that
    \[
        a = a^\dagger = h^{-1} \circ d(a) \circ h.
    \]
    We need to find an isomorphism $f\colon (c,h) \to (c', h')$ such that 
    \[
        a \stackrel{?}{=} f^\dagger \circ f = h^{-1} \circ d(f) \circ h' \circ f.
    \]
    Indeed, this can always be achieved by setting $c'=c$, $f = \id_c$, and $h' = h \circ a$.
    It just remains to check that $h'$ is indeed a valid Hermitian fixed point on $c$.
    For this we consider 
    \[
        d(h') \circ \eta_c = d(a) \circ d(h) \circ \eta_c = d(a) \circ h = h \circ a = h'.
    \]
\end{proof}

We now begin to construct the unit and counit for the adjunction between $\mrT$ and $\Herm$.

\begin{defn}\label{defn:THerm->id}
    For every anti-involutive category $(\calC, d, \eta)$ we define an involutive functor
    \[
        (K_\calC, \varphi_\calC) \colon \mrT(\Herm(\calC, d, \eta)) \too (\calC, d, \eta)
    \]
    by letting $K_\calC$ be the functor $(c, h) \mapsto c$ and $f \mapsto f$,
    and letting $\varphi_\calC\colon  K_\calC^{op} \circ \dagger_{\mrT(\Herm(\calC))} \cong d \circ K_\calC$ be the natural transformation given by 
    \[
        \varphi_{(c,h)} := (h\colon c \to d(c)).
    \]
\end{defn}

\begin{lem}\label{lem:counit-equivalence}
    The involutive functor $(K_\calC, \varphi_\calC)$ is well-defined, natural in $\calC$, and it is an equivalence of anti-involutive categories onto the full subcategory $\calC^\exfix \subset \calC$.
\end{lem}
\begin{proof}
    To check that $\varphi$ is indeed a natural transformation we need to consider for each morphism $f\colon (c_1,h_1) \to (c_2, h_2)$ the square:
    \[
        \begin{tikzcd}[column sep = huge]
            c_1 \ar[r, "\varphi_{(c_1,h_1)} = h_1"] &
            d(c_1) \\
            c_2 \ar[r, "\varphi_{(c_2,h_2)} = h_2"] \ar[u, "f^\dag"] & 
            d(c_2) \ar[u, "d(f)"'].
        \end{tikzcd}
    \]
    This indeed commutes by the definition of $f^\dag$.
    This natural transformation further has to satisfy that for each $(c,h)$ the square
    \[
        \begin{tikzcd}[column sep = large]
            K_\calC(c,h) \ar[d, "\eta_{K_\calC(x)}", swap] \ar[r, equal] & (K_\calC \circ \dag^{op} \circ \dag)(c,h) \ar[d, "\varphi_{(c,h)} = h"]\\
            (d^{op}\circ d\circ K_\calC)(c,h) \ar[r, "d^{op}(\varphi_{(c,h)})"] & (d^{op} \circ K_\calC^{op} \circ \dag)(c, h) 
        \end{tikzcd}
    \]
    commutes.
    Upon closer inspection this is exactly the triangle that commutes because $h$ is is a Hermitian fixed point.
    
    It follows from the construction that $(K_\calC, \varphi_\calC)$ is natural in $\calC$.
    Moreover, $K_\calC$ is certainly fully faithful and essentially surjective onto the subcategory of $\mathcal{C}$ that admit a Hermitian fixed point, so by Lemma \ref{lem:involutive-equivalence} $(K_\calC, \varphi_\calC)$ is an equivalence in $\ICat$.
\end{proof}

\begin{defn}\label{defn:U}
    For every dagger category $(\calD, \dagger)$ we define a dagger functor
    \[
        U_{\mathcal{D}} \colon \calD \too \Herm(T(\calD))
    \]
    by sending $x$ to $(x,\id)$ and $f\colon x \to y$ to $f\colon (x,\id) \to (y,\id)$.
\end{defn}

The construction of $U_{\mathcal{D}}$ is well-defined and natural in $\calD$.
Moreover, $U_{\mathcal{D}}$ is always fully faithful and essentially surjective.
However, the more subtle question is when $U$ is surjective up to unitaries.

\begin{lem}\label{lem:unit-equivalence}
    The functor $U_\calD$ is an equivalence of dagger categories if and only if $\calD$ is indefinite.
\end{lem}
\begin{proof}
    As noted before $U_{\mathcal{D}}\colon \calD \too \Herm(T(\calD))$ is always an equivalence of categories,
    so by Lemma \ref{lem:dagger-equivalence} we only need to check when it is surjective up to unitaries.
    Suppose $(y,h) \in \Herm(T(\calD))$ is an object that is unitarily isomorphic to some object $(x, \id_x)$ in the essential image.
    Then we have an isomorphism $f\colon (y, h) \to (x,\id_x)$ satisfying $\id_{(y,h)} = f^* \circ f $.
    (Here we write $*$ for the dagger on $\Herm(T(\calD))$ to distinguish it from the dagger $\dag$ on $\calD$.)
    Spelling out the definition we see that $\id_{(y,h)} = f^* \circ f = (h^{-1} \circ f^\dag \circ \id_x) \circ f$, 
    or equivalently $h = f^\dag \circ f$.
    Therefore $U_\calD$ is surjective up to unitaries if and only if every self-adjoint automorphism $h$ can be written as $f^\dag \circ f$ with $f$ invertible,
    i.e.~if and only if $\calD$ is indefinite.
\end{proof}

Recall that a $2$-adjunction is a $\Cat_1$-enriched adjunction, i.e.~an adjunction for which the unit and counit satisfy the triangle identities \emph{strictly}. \cite{kellystreet}

\begin{thm}\label{thm:2-adjunction}
    The functors $U$ and $K$ exhibit a $2$-adjunction:
    \[
        \mrT \colon \dCat \rightleftarrows \ICat \,:\! \Herm
    \]
    and this restricts to a biequivalence between $\ICat^\exfix$ and the full sub-$2$-category $\dCat^{\rm indef}$ on the indefinite complete dagger categories.
\end{thm}
\begin{proof}
    To establish the $2$-adjunction $\mrT \dashv \Herm$ with unit $U$ and counit $K$ we need to check the triangle identities.
    The first identity concerns for each $(\calC, d, \eta) \in \ICat$ the composite functor
    \[
        \Herm(\calC) \xrightarrow{U_{\Herm(\calC)}} 
        \Herm(\mrT(\Herm(\calC))) \xrightarrow{\Herm(K_\calC, \varphi_\calC)}
        \Herm(\calC).
    \]
    The first functor sends $(x, h)$ to $((x,h), \id)$
    and the second functor sends this to $(x, (\varphi_\calC)_{(x,h)} \circ K_\calC(\id)) = (x, h \circ \id) = (x,h)$.
    By construction the composite functor also the identity on morphisms.
    
    The second identity concerns for each $(\calD,\dag) \in \dCat$ the composite functor
    \[
        \mrT(\calD) \xrightarrow{\mrT(U_\calC)} 
        \mrT(\Herm(\mrT(\calD))) \xrightarrow{K_{\mrT(\calD)}, \varphi_{\mrT(\calD)}}
        \mrT(\calD).
    \]
    The first functor sends $x$ to $(x, \id_x)$ and the second functor sends this to $x$. 
    On morphisms the composite is also the identity.
    It remains to check that the involutive data of the composite functor is trivial.
    For the first functor this holds by definition.
    For the second functor we have $\varphi_{\mrT(\calD)}(x,h) = h$, but since we are applying this to the object $(x, \id_x)$, it is also trivial.
    
    Finally, we would like to show that this adjunction restricts to a biequivalence between $\dCat^\indef$ and $\ICat^\exfix$.
    The adjunction does restrict because $\Herm(\calC)$ is always indefinite and $\mrT(\calD)$ always has fixed-point structures.
    The restriction is a biequivalence by lemma \ref{lem:unit-equivalence} and lemma \ref{lem:counit-equivalence}, which state that on these subcategories the unit and counit become equivalences.
\end{proof}

\section{Choosing positive Hermitian structures}

The goal of this section is to prove the main theorem, which relates dagger categories with anti-involutive categories.
In the previous section, we aqcuired an understanding of the relationship between anti-involutive categories and indefinite categories
We thus need to discuss how to obtain dagger categories that are not indefinite from categories with anti-involution.
To achieve this we will restrict the collection of `allowed' Hermitian fixed points on the Hermitian completion to a smaller class of `positive' Hermitian fixed points.
This will yield a smaller dagger subcategory for which the underlying category with anti-involution is equivalent.
For example, to get the dagger category $\Hilbfd$ we take the Hermitian completion of $\Vectfd$ and then restrict to the subclass of Hermitian fixed points that are positive definite as Hermitian pairings.

So let $(\mathcal{C}, d, \eta)$ be a category with anti-involution.
For $P$ any subset of the collection of all Hermitian fixed points in $\mathcal{C}$, let $\mathcal{C}_P \subseteq \operatorname{Herm} \mathcal{C}$ denote the full subcategory on all $(c,h) \in P$.
Here $P$ stands for `positive' to remind us of the typical situation in vector spaces in which we wanted to restrict the Hermitian fixed points to the positive definite ones to obtain the dagger category of Hilbert spaces.
The dagger from $\Herm \calC$ restricts to a dagger on $\mathcal{C}_P$.

We are interested in understanding how many dagger categories we can get by this procedure that are not unitarily equivalent.
For this, first note that if $P \subseteq P'$, inclusion $\mathcal{C}_P \to \mathcal{C}_{P'}$ defines a dagger functor, which is fully faithful.
However, even when $P \neq P'$ this inclusion can still be a unitary equivalence.
Namely, we will show that adding \emph{transfers} of Hermitian fixed points to $P$ does not change the unitary equivalence class of $\mathcal{C}_P$:

\begin{defn}
    Given a Hermitian fixed point $h\colon  c \to dc$ and an isomorphism $g\colon c' \to c$, the \emph{transfer} of $h$ by $g$ is the Hermitian fixed point defined on $c'$ by $d(g) \circ h \circ g$.
\end{defn}

Note that this is indeed a Hermitian fixed point because the following diagram commutes
\[
\begin{tikzcd}
    c' \ar[r,"g"] \ar[rrr,bend left,"d(g) \circ h \circ g"] \ar[d,"{\eta_{c'}}"] & c \ar[r,"h"] \ar[d,"{\eta_{c}}"] & dc \ar[r,"dg"] \ar[d,equals] & dc' \ar[d,equals]
    \\
    d^2 c' \ar[r,"d^2 g"] \ar[rrr,bend right,"d(d(g) \circ h \circ g)"] & d^2 c \ar[r,"dh"] & dc \ar[r,"dg"] & dc'
\end{tikzcd}.
\]

\begin{lem}
\label{lem:positivity}
    Two objects $(c,h), (c',h') \in \Herm \mathcal{C}$ are unitarily isomorphic if and only if $h'$ is a transfer of $h$.
\end{lem}
\begin{proof}
    An isomorphism $\alpha\colon  (c',h') \to (c,h)$ is unitary if and only if 
    \[
    \alpha^{-1} = \alpha^\dagger \stackrel{\text{defn}}{=} h'^{-1} \circ d\alpha \circ h.
    \]
    This happens if and only if $h'= d \alpha \circ h \circ \alpha$.
\end{proof}

\begin{defn}
   Given a category $\calC$, let $\pi_0(\calC)$ denote the collection of isomorphism classes of objects.
   If $\calC$ is additionally a dagger category, let $\pi_0^U(\calC)$ denote the collection of unitary isomorphism classes of objects.
\end{defn}

We can rephrase the above lemma by saying that $\pi_0^U(\Herm \mathcal{C})$ is the collection of Hermitian fixed points $(c,h)$ modulo transfer.
Note that a dagger functor is unitarily essentially surjective if and only if it is surjective on $\pi_0^U$.
In particular, if $P$ is a collection of Hermitian fixed points and $P'$ is the closure of $P$ under transfers, then $\mathcal{C}_P \to \mathcal{C}_{P'}$ is unitarily essentially surjective and hence an equivalence of dagger categories.
Therefore we can assume without loss of generality that $P$ is closed under transfers.

Now let $P_c \subseteq P$ denote the subset of Hermitian fixed points on the object $c$.
Then we will want to require that $P_c \neq \emptyset$, so that every object has `some positive Hermitian structure'.
This will additionally ensure that $\mathcal{C}_P \to \Herm \mathcal{C}$ is essentially surjective.

This discussion motivates us to make the following definition.

\begin{defn}
\label{defn:positivity}
    Let $(\mathcal{C}, d, \eta)$ be a category with anti-involution.
    A \emph{positivity notion} on $\mathcal{C}$ is a collection of subsets 
    \[
    P = \{P_c \subset \Hom_\calC(c, d(c)): c \in \operatorname{obj} \mathcal{C} \}
    \]
    such that:
    \begin{itemize}
        \item each $P_c$ is non-empty,
        \item each $(h\colon c \to d(c)) \in P_c$ is a Hermitian fixed point,
        \item $P$ is closed under transfer.
    \end{itemize}
\end{defn}

\begin{rem}
    A necessary and sufficient condition for an anti-involutive category to admit a positivity notion is that every object admits some Hermitian fixed point structure.
\end{rem}

\begin{ex}
If $\mathcal{C}$ is a category with anti-involution in which every object admits some Hermitian fixed points, we can take $P$ to consist of all Hermitian fixed points.
This is a positivity notion and $\mathcal{C}_P = \Herm \mathcal{C}$.
\end{ex}

\begin{cor}
    Positivity notions on an anti-involutive category $\calC$ are in bijection with subsets $[P] \subset \pi_0^U(\Herm \mathcal{C})$ such that the composite 
    \[
    [P] \subset \pi_0^U(\Herm \mathcal{C}) \to \pi_0(\calC)
    \]
    is surjective.
\end{cor}
\begin{proof}
    It follows immediately by the lemma above that a subset $[P] \subset \pi_0^U(\Herm \mathcal{C})$ is equivalent to a choice of Hermitian fixed points on some collection of objects of $\mathcal{C}$ that is additionally closed under transfer.
    The condition that the given composite is surjective is equivalent to requiring that for every object $c$ there exists an isomorphic object $c'$ and a Hermitian fixed point $h\colon  c' \to dc'$ such that $(c',h\colon c' \to dc') \in P_{c'}$.
    In case such $(c',h)$ exists, we also get that $P_c \neq \emptyset$ by transferring $h$ to $c$.
    Conversely it is clear that the desired composite is surjective if $P_c \neq \emptyset$ for all $c$.
\end{proof}

\begin{ex}\label{ex:positivity-structure-on-dagger-cat}
    Recall that if $(\calD, \dag)$ is a dagger category, a Hermitian fixed point on the anti-involutive category $\mrT(\calD, \dag)$ is the same as a self-adjoint automorphism $h\colon c \to c^\dag = c$.
    There is a canonical positivity notion on $\mrT(\calD, \dag)$, which is defined by
    \[
        P_c := \{ h\colon c \to c \;|\; \text{ There is an isomorphism } f\colon c \to c' \text{ with } h = f^\dag \circ f \}.
    \]
    This is exactly smallest positivity notion that contains all identities.%
    \footnote{
        In the published version of this paper we incorrectly required $f$ to be an automorphism rather than an isomorphism, which does not define a positivity notion as it is not always closed under transfers. 
    }
\end{ex}

\begin{rem}\label{rem:positive}
    An endomorphism $h\colon c \to c$ in a dagger category (or $C^*$-category) is called \emph{positive} if there is a morphism $f\colon c \to c'$ with $h= f^\dag \circ f$.
    The set $P_c \subset \hom_\calC(c,c)$ from example \ref{ex:positivity-structure-on-dagger-cat} is contained in the set of positive automorphisms.
    However, it is not true in general that every positive automorphism is in $P_c$:
    it might happen that some automorphism $h\colon c \to c$ can be written $h = f^\dag \circ f$ for some morphism $f\colon c \to c'$, but that $f$ cannot be chosen to be invertible.
    (A concrete counterexample can be constructed as a specific subexample of Example \ref{ex:hilb}.)
\end{rem}

\begin{ex}
\label{ex:hilb}
We study the case of finite-dimensional vector spaces with $d = \overline{(.)}^*$ as before.
Note that $\pi_0^U(\Herm(\Vectfd_\C)) = \N \times \N$ given by the signature of the corresponding Hermitian pairing. 
Here the signature of $(V,(.,.))$ is the pair $(p,q)$ so that there exists an orthonormal basis $\{e_1, \dots, e_{p+q}\}$ of $V$ with
\[
(e_i, e_i) = 1 \quad (e_j,e_j) = -1
\]
for $i \leq p$ and $j > p$.
The forgetful map $\pi_0^U(\Herm (\Vectfd_\C)) \to \pi_0(\Vectfd_\C) = \mathbb{N}$ is addition.

We provide some examples of positivity notions on this category with anti-involution.
One example is to take $P_V$ to be the collection of positive definite Hermitian inner products on $V$.
This is a positivity notion because every finite-dimensional vector space admits a positive definite inner product.
The resulting dagger category $\mathcal{C}_P$ is the dagger category of finite-dimensional Hilbert spaces.
Another example is to take $P_V$ to consist of all Hermitian inner products in which case $\mathcal{C}_P$ is the dagger category of finite-dimensional vector spaces with arbitrary Hermitian inner products.

There are also many more unusual positivity notions on $\Vectfd_\C$.
Namely, for every dimension $d$ we can separately specify a nonempty collection of signatures $(p,q) \in \N \times \N$ such that $p+q =d$ we allow for Hermitian forms.
Any such choice gives a dagger category $\mathcal{C}_P$ and two different choices are not dagger equivalent compatibly with the map to $\Herm (\Vectfd_\C)$.
Note that it would be reasonable to restrict the allowed positivity notions further by requiring compatibility with tensor products or direct sums, but we will not pursue this further here.
\end{ex}

\begin{obs}\label{obs:positive-functor}
    Given an involutive functor 
    $(F, \varphi)\colon \calC \to \calD$
    and positivity notions $P$ on $\calC$ and $Q$ on $\calD$ the following are equivalent:
    \begin{enumerate}
        \item For all $(h\colon c \to d(c)) \in P_c$, we have 
        \[
        \Herm F (c,h) = (\varphi_c \circ F(h)\colon F(c) \to d(F(c))) \in Q_{F(c)}.
        \]
        \item The map $\pi_0^U( \Herm \calC) \to \pi_0^U(\Herm \calD)$ induced by $F$ sends $[P]$ to a subset of $[Q]$.
    \end{enumerate}
\end{obs}

\begin{defn}
    The $2$-category $\PCat$ has as objects anti-involutive categories equipped with a positivity notion.
    Morphisms are involutive functors that intertwine the positivity notions in the sense of equivalent conditions in observation \ref{obs:positive-functor}.
    The $2$-morphisms are the same as in $\ICat$.
\end{defn}

\begin{rem}
    Note that the forgetful functor $\PCat \to \ICat$ is well-behaved:
    Positivity notions can be transported along equivalences of categories and they can be restricted along fully faithful functors.
    Therefore the forgetful functor has lifts for equivalences. 
    Moreover, if we restrict to fully faithful functors as morphisms in both $2$-categories, then the functor
    $\PCat^{\rm ff} \to \ICat^{\rm ff}$ 
    is equivalent to the Grothendieck construction of the functor 
    $(\ICat^{\rm ff})^{op} \to \mathrm{PoSet}$
    that sends an anti-involutive category to its poset of possible positivity notions.
    
    We can also easily characterise the equivalences in the $2$-category $\PCat$.
    An involutive functor $(F, \varphi)\colon  (\calC, P) \to (\calD, Q)$ is an equivalence in $\PCat$, if and only if $F$ is an equivalence of categories (and hence $(F, \varphi)$ is an equivalence in $\ICat$ by lemma \ref{lem:involutive-equivalence}),
    and moreover the induced map of sets $[P] \to [Q]$ is surjective. 
    (This map is automatically injective since $F$ is fully faithful.)
    The latter condition says that every positive Hermitian fixed point in $(\calD, Q)$ is (up to transfer) of the form $\varphi_c \circ F(h)$ for $h\colon c \to d(c)$ a positive Hermitian fixed point in $(\calC, P)$. 
\end{rem}

\begin{thm}
    \label{thm:biequivalence}
    Equipping $\mrT(\calD, \dag)$ with the positivity notion from example \ref{ex:positivity-structure-on-dagger-cat}
    defines a lift $\mrT_\pos$:
    \[
    \begin{tikzcd}
        & \PCat \ar[d, "{\rm forget}"] \\
        \dCat \ar[ru, "{\mrT_\pos}"] \ar[r, "\mrT"] & \ICat.
    \end{tikzcd}
    \]
    This $2$-functor $\mrT_{\pos}$ is a biequivalence.
\end{thm}

\begin{proof}
    We define an inverse $2$-functor
    \[
        \Herm_{\pos}\colon \PCat \to \dCat
    \]
    by declaring $\Herm_{\pos}(\calC, d, \eta, P) \subset \Herm(\calC, d, \eta)$ to be the full sub-$\dag$-category on those Hermitian fixed points $(h\colon c \to d(c))$ where $h \in P_c$.
    In other words, $\Herm_{\pos}$ is the subcategory of Hermitian fixed points which are positive.
    This is a well-defined $2$-functor because $\Herm$ is and because $1$-morphisms in $\PCat$ preserve the positivity notions by definition.
    
    The functor $U\colon \calD \to \Herm(\mrT(\calD))$ from definition \ref{defn:U} that sends $x$ to $(x, \id_x)$,
    restricts to a functor $U_{\pos}\colon \calD \to \Herm_{\pos}(\mrT_{\pos}(\calD))$ since $(x, \id_x)$ is always a positive Hermitian fixed point in $\mrT_{\pos}(\calD)$.
    Therefore this defines a natural transformation
    $U \colon \id_{\dCat} \to \Herm_{\pos} \circ \mrT_{\pos}$
    and as observed below definition \ref{defn:U} the dagger functor
    \[
        U_{\pos}\colon \calD \to \Herm_{\pos}(\mrT_{\pos}(\calD))
    \]
    is always fully faithful (and essentially surjective).
    We would like to show that it is surjective up to isometry.
    Let $(h\colon x \to x)$ be some object in $\Herm_{\pos}(\mrT_{\pos}(\calD))$.
    Since $h$ is a positive Hermitian fixed point in $\mrT_{\pos}(\calD)$, we can write it as $h = f^\dag \circ f$ for some isomorphism $f \colon x \to y$.
    In other words, $h$ is the transfer of $(\id\colon y \to y)$ along $f$.
    By lemma \ref{lem:positivity} this means that there is a unitary isomorphism $(y,\id_y) \cong (x, h)$, showing that $U_\pos$ is surjective up to unitary isomorphism.
    Therefore $U_{\pos}$ is a dagger equivalence and hence an equivalence in the $2$-category $\dCat$.
    
    Finally, consider the involutive functor
    \[
        (K_\mathcal{C}, \varphi_\calC) \colon \mrT(\Herm(\calC, d, \eta)) \too (\calC, d, \eta)
    \]
    that we constructed naturally for all $(\calC, d, \eta) \in \ICat$ in definition \ref{defn:THerm->id}.
    Given a positivity notion $P$ on $\calC$, we also get a positivity notion on $\mrT_{\pos}(\Herm_{\pos}(\calC, d, \eta, P))$.
    We would like to show that $K_\calC$ preserves positivity notions.
    A positive Hermitian fixed point in $\mrT_{\pos}(\Herm_{\pos}(\calC, d, \eta, P))$ is of the form $p = f^\dag \circ f$ for some isomorphism $f\colon (x, h) \to (y,h')$.
    Here $h\colon x \to d(x)$ and $h'\colon y \to d(y)$ are Hermitian fixed points in $\calC$ that are positive in the sense that they are contained in $P$.
    Using the definition of the dagger in $\Herm(\calC)$ we can write this as $p = (h^{-1} \circ d(f) \circ h') \circ f$.
    In order to show that $(K_\mathcal{C}, \varphi_\calC)$ respects the positivity notions, we use condition (1) of observation \ref{obs:positive-functor}, which says that $\varphi_{(x,h)} \circ K_\mathcal{C}(p)$ must be positive.
    Using $\varphi_{(x,h)}= h \colon x \to d(x)$ we see that
    \[
        \varphi_{(x,h)} \circ K_\mathcal{C}(p) = h \circ (h^{-1} \circ d(f) \circ h' \circ f) = d(f) \circ h' \circ f.
    \]
    is the transfer of $h'$ along $f$, and since $h'$ was positive, so is this.
    Therefore $(K_\mathcal{C}, \varphi_\calC)$ defines a natural morphism
    \[
        \mrT_{\pos}(\Herm_{\pos}(\calC, d, \eta, P)) \too (\calC, d, \eta, P)
    \]
    of anti-involutive categories with positivity notions.
    We already observed in lemma \ref{lem:counit-equivalence} that $(K_\mathcal{C}, \varphi_\calC)$ is an equivalence in $\ICat$.
    For it to also be an equivalence in $\PCat$ we need to check that every positive fixed point in $(\calC, d, \eta, P)$ is hit (up to transfer) by $(K_\mathcal{C}, \varphi_\calC)$.
    We saw above that every morphism $d(f) \circ h' \circ f$ can be written as $\varphi_{(x,h)} \circ K_\mathcal{C}(p)$.
    Setting $f=\id_x$ we see that indeed every positive fixed point $h'$ can be hit by this.
\end{proof}

\begin{cor}\label{cor:dagger-functors-as-fixedpoints}
        Let $(\calC, \dag)$ and $(\calD, \dag)$ be two dagger categories.
    Then $F \mapsto \dag_\calD \circ F \circ \dag_\calC$ defines an anti-involution on the category of all functors $\calC \to \calD$.
    The inclusion of the dagger functors into the fixed points 
    \[
        \Fun^\dag((\calC, \dag), (\calD, \dag)) \hookrightarrow 
        \left(\Fun(\calC, \calD)\right)^{\rm fix}
    \]
    is fully faithful and its essential image consists of those functors that preserve the positivity notions.
\end{cor}
\begin{proof}
    We have seen in \ref{functorsfix} that given two anti-involutive categories $(\calC, d_\calC), (\calD, d_\calD)$ there is an anti-involution $d$ on $\Fun(\calC, \calD)$ given by the expression $F \mapsto d_\calD \circ F \circ d_\calC$.
    Its fixed points are the category of which objects are involutive functors and morphisms are involutive natural transformations.
    Specializing to the case where the anti-involutive categories come from dagger categories, we see that
    \[
    \left(\Fun(\calC, \calD)\right)^{\rm fix} \cong \Hom_{\ICat}(\mrT\calC, \mrT \calD).
    \]
    The corollary now follows directly from the main theorem.
\end{proof}

\section{Applications to unitary topological field theory}
\label{sec:outlook}


In this section, we will outline how this work applies to the question of how to define unitary topological quantum field theory.
Recall that Atiyah \cite{atiyahtft} introduced the notion of a topological quantum field theory.
A topological quantum field theory (TQFT) is defined as a symmetric monoidal functor $(\Bord_{d,d-1}, \sqcup) \to (\Vect_\C, \otimes)$ from the oriented bordism category to vector spaces.
With the purpose of defining unitary TQFT, he also introduced a Hermitian axiom.
This required Hilbert space pairings in such a way that simultaneously reversing in- and output and orientation-reversing bordisms amounts to taking adjoints of operators.
A more precise formulation is: \cite{baezdagger, turaev}

\begin{defn}
A \emph{$d$-dimensional unitary TQFT} is a symmetric monoidal dagger functor
\[
(\Bord_{d,d-1}, \sqcup) \to (\Hilb, \otimes).
\]
\end{defn}

In \cite[Definition 4.14]{freedhopkins}, the authors define a reflection structure on a TQFT to be $\Z/2$-equivariance data for certain $\Z/2$-actions $\overline{(.)}$ on the domain and target category.
They define reflection positive TQFTs as reflection TQFTs preserving a certain positivity notion.
Our approach in this paper is strongly motivated by Freed-Hopkins.
It is shown in \cite[Section 2.2]{luukthesis} that reflection structures on a TQFT are equivalent to anti-involutive structures for the anti-involution $\overline{(.)}^*$.
The following corollary thus makes precise the relationship between reflection positive and unitary TFT.

In \cite[Appendix A]{luukspinstatistics}, it is shown that our main theorem \ref{thm:biequivalence} generalises to (symmetric/braided) monoidal categories. In particular, we have the following analogue of Corollary \ref{cor:dagger-functors-as-fixedpoints}:
\begin{cor}\label{cor:sm-fix}
    Let $(\calC, \dag, \otimes)$ and $(\calD, \dag, \otimes)$ be two symmetric monoidal dagger categories.
    Then $F \mapsto \dag_\calD \circ F \circ \dag_\calC$ defines an anti-involution on the category of symmetric monoidal functors $\calC \to \calD$.
    The inclusion of the symmetric monoidal dagger functors into the fixed points 
    \[
        \Fun^{\otimes,\dag}((\calC, \dag, \otimes), (\calD, \dag, \otimes)) \hookrightarrow 
        \left(\Fun^\otimes(\calC, \calD)\right)^{\rm fix}
    \]
    is fully faithful and its essential image consists of those functors that preserve the positivity notions.
\end{cor}
\begin{proof}
    Following the proof of Lemma \ref{functorsfix}, we can construct an anti-involution on the category $\Fun^\otimes((\calC, \otimes), (\calD, \otimes))$ of symmetric monoidal functors by setting
    \[
        dF := \dagger_\calD \circ F \circ \dagger_\calC,
    \]
    which in the case at hand will square to $d^2 F = F$ because $\calC$ and $\calD$ are dagger categories.
    Therefore we may set $\eta = \id_F$.
    As in Lemma \ref{functorsfix}, the groupoid of Hermitian fixed points of $(d, \eta)$ is the groupoid of anti-involutive symmetric monoidal functors as defined in \cite[Definition 2.2.12]{luukthesis}.
    Now \cite[Theorem A.8]{luukspinstatistics} shows that the groupoid of symmetric monoidal dagger functors is equivalent to the groupoid of those anti-involutive symmetric monoidal functors that preserve the positivity notions.
\end{proof}

As an example of how Corollary \ref{cor:sm-fix} can be applied to TQFTs, we will consider the case of $2$-dimensional unitary TQFTs.
First, let us recall the folk-theorem that classifies them as commutative Frobenius algebras.

\begin{defn}
    We define the groupoid of commutative Frobenius algebras over $\C$, $\mathrm{CFrob}_\C$ to have objects
    \[
        (A, \mu\colon A \otimes A \to A, \nu\colon 1 \to A, \Delta \colon A \to A \otimes A, \varepsilon\colon A \to 1)
    \]
    where $(A, \mu, \nu)$ is a commutative algebra, $(A, \Delta, \varepsilon)$ is a cocommutative coalgebra, and they satisfy the Frobenius axiom
    \[
        (\mu \otimes \id_A) \circ (\id_A \otimes \Delta) 
        = \Delta \circ \mu
        = (\id_A \otimes \mu) \circ (\Delta \otimes \id_A).
    \]
    Morphisms in this groupoid are maps $A \to B$ that are both algebra and coalgebra homomorphisms.
    (These are necessarily invertible.)
\end{defn}

\begin{thm}[\cite{abramstqft}, also see \cite{kock2dtqft}]\label{thm:folk-2d-tqft}
    There is an equivalence of groupoids
    \[
        \Fun^\otimes(\Bord_{1,2}, \Vect_\C) \simeq \mathrm{CFrob}_\C
    \]
    defined by sending $\calZ \colon \Bord_{1,2} \to \Vect_\C$ to $\calZ(S^1)$, equipped with the Frobenius structure given by
    \[
        (\calZ(S^1), \mu = \calZ(\pants), \nu = \calZ(\disk), \Delta = \calZ(\copants), \varepsilon = \calZ(\codisk)).
    \] 
\end{thm}

    To study $2$-dimensional unitary TQFTs we equip $\Bord_{1,2}$ with the symmetric monoidal anti-involution $\dagger$ that reverses bordisms and
    $\Vect_\C$ with the symmetric monoidal anti-involution $\overline{(-)}^*$.
    Recall that this anti-involution $\Vect_\C$ corresponds to the indefinite dagger category $\Herm$.
    From Corollary \ref{cor:sm-fix}, we get an induced anti-involution on $\Fun^\otimes(\Bord_{1,2}, \Vect_\C)$ defined by
    \[
        \calZ \longmapsto \overline{(-)^*} \circ \calZ \circ \dagger.
    \]
    The equivalence in Theorem \ref{thm:folk-2d-tqft} is equivariant with respect to this involution, if we equip $\mathrm{CFrob}_\C$ with the involution:
    \[
        d\colon (A, \mu, \nu, \Delta, \varepsilon) \longmapsto 
        (\overline{A}^*, \overline{\Delta}^*, \overline{\varepsilon}^*, \overline{\mu}^*, \overline{\nu}^*),
    \]
    with $\eta\colon \id \cong d^2$ given by the same natural isomorphism as in $(\Vect_\C, \overline{(-)}^*)$.
    Here we implicitly used that $\overline{(-)^*}$ is symmetric monoidal to get isomorphisms $\overline{A \otimes A}^* \cong \overline{A}^* \otimes \overline{A}^*$.
    Passing to fixed points we get an equivalence
    \[
        \Fun^{\otimes,\dagger}(\Bord_{1,2}, \Herm) \simeq \mathrm{CFrob}_\C^{\rm fix}.
    \]
    A fixed point on the right is a commutative Frobenius algebra $(A, \mu, \nu, \Delta, \varepsilon)$ with an isomorphism $\alpha\colon A \cong \overline{A}^*$ of Frobenius algebras that ``squares'' to the identity.
    We can encode the isomorphism in terms of a non-degenerated sesquilinear pairing $\langle.,. \rangle \colon \overline{A} \otimes A \to \C$  such that
    \[
    \langle a,b \rangle = \overline{\langle b,a \rangle}.
    \]
    That $\alpha$ is an isomorphism of Frobenius algebras then means that with respect to the pairing $\langle .,. \rangle$, $\mu$ is adjoint to $\Delta$ and $\nu$ is adjoint to $\varepsilon$.
    Because of this it will suffice to only encode the algebra structure and the pairing.
    We make the following definition, see e.g.~\cite[Definition 3.3]{vicary-Frobenius}:

\begin{defn}
    A \emph{Hermitian commutative Frobenius algebra} is a tuple $(A, \mu, \nu, \langle .,. \rangle)$ of a commutative algebra $(A, \mu, \nu)$ and a Hermitian pairing $\langle .,. \rangle$ on $A$ such that 
    \[
        (\id \otimes \mu) \circ (\id \otimes \mu^\dagger) = \mu^\dag \circ \mu = (\id \otimes \mu) \circ (\mu^\dagger \otimes \id)
    \]
    where $(-)^\dagger$ denotes the adjoint operator with respect to the pairing.
    A Hermitian commutative Frobenius algebra is called \emph{unitary} if $\langle .,. \rangle$ is positive definite.
    A morphism of Hermitian commutative Frobenius algebras is an isometry that is also an algebra homomorphism.
\end{defn}

In summary, we obtain the following:

\begin{cor}
    There is an equivalence between $2$d unitary TFTs and unitary commutative Frobenius algebras.
\end{cor}
\begin{proof}
    If $(A, \mu, \nu, \langle .,. \rangle)$ is a Hermitian commutative Frobenius algebra, then $(A, \mu, \nu, \mu^\dagger, \nu^\dagger)$ is a commutative Frobenius algebra. Indeed, $(A,\mu^\dagger,\nu^\dagger)$ is cocommutative coalgebra because $(A, \mu, \nu)$ is a commutative algebra. 
    Clearly $(A, \mu, \nu, \mu^\dagger, \nu^\dagger)$ is a fixed-point of the anti-involution on $\mathrm{CFrob}_\C$ defined above.
    Conversely all fixed points are of this form and we in fact have an equivalence of groupoids.
    Combining the classification of ordinary $2$-dimensional TQFTs \ref{thm:folk-2d-tqft} with Corollary \ref{cor:sm-fix}, we see that there is an equivalence between anti-involutive TQFTs and Hermitian commutative Frobenius algebras.
    Therefore the corollary follows from Corollary \ref{cor:sm-fix} after realizing that the TQFT preserves positivity notions if and only if $A$ is a Hilbert space.
\end{proof}

\begin{rem}
    We illustrated how our theorem allows us to take the computation of non-unitary TFTs as a black box and from it compute the groupoid of unitary TFTs.
    This computation has been done by hand in the literature, see \cite{durhuus1994classification,sawin1995direct, zhu2dhermitiantqft}.
\end{rem}

\begin{rem}
    In \cite{freedhopkins} Freed--Hopkins define and then classify invertible fully extended unitary (or in their setting reflection positive) TQFTs, but how to define general fully extended unitary TQFTs remained open.
    In \cite{higherdagger}, based on the current article, a proposal for a definition of a dagger $n$-category is given, together with the construction of a bordism dagger $n$-category.
    This leads to a definition of a fully local unitary TQFT with values in a target dagger $n$-category.
\end{rem}

\bibliographystyle{unsrt}
\bibliography{literature}

@article{vicary,
    title = {Completeness of $\dagger$-categories and the complex numbers},
    author = {Vicary, Jamie},
	journal = {J. Math. Phys.},
	year = {2011}, 
	volume = {52}, 
	pages = {82-104}
}

@article{vicary-Frobenius,
    title = {Categorical Formulation of Finite-Dimensional Quantum Algebras},
    author = {Vicary, Jamie},
	journal = {Communications in Mathematical Physics},
	year = {2011}, 
	volume = {304}, 
	pages = {765--796}
}

@phdthesis{srinivasan2023dagger,
  title={Dagger linear logic and categorical quantum mechanics},
school={University of Calgary},
  author={Srinivasan, Priyaa Varshinee},
  year={2023}
}

@article{atiyahtft,
  title={Topological quantum field theory},
  author={Atiyah, Michael},
  journal={Publications Math{\'e}matiques de l'IH{\'E}S},
  volume={68},
  pages={175--186},
  year={1988}
}

@article{beggs,
  title={Bar categories and star operations},
  author={Beggs, Edwin and Majid, Shahn},
  journal={Algebras and Representation Theory},
  volume={12},
  number={2-5},
  pages={103--152},
  year={2009},
  publisher={Springer}
}

@article{egger,
  title={On involutive monoidal categories},
  author={Egger, Jeffrey},
  journal={Theory and Applications of Categories},
  volume={25},
  number={14},
  pages={368--393},
  year={2011},
  publisher={Mount Allison University}
}

@article{freedhopkins,
  title={Reflection positivity and invertible topological phases},
  author={Freed, Daniel and Hopkins, Michael},
  journal={Geometry \& Topology},
  volume={25},
  number={3},
  pages={1165--1330},
  year={2021},
  publisher={Mathematical Sciences Publishers}
}

@misc{henriquespenneys,
    url = {https://arxiv.org/abs/2004.08271},
  title={Representations of fusion categories and their commutants},
  author={Henriques, Andr{\'e} and Penneys, David},
  note={Available at \href{https://arxiv.org/abs/2004.08271}{arXiv:2004.08271}},
  year={2020}
}

@misc{henriques2023,
  title={Unitary anchored planar algebras},
  author={Henriques, Andr{\'e} and Penneys, David and Tener, James},
  note={Available at \href{https://arxiv.org/abs/2301.11114}{arXiv:2301.11114}},
  year={2023}
}

@phdthesis{hessethesis,
  title={Group Actions on Bicategories and Topological Quantum Field Theories},
  author={Hesse, Jan},
  year={2017},
  school={Staats-und Universit{\"a}tsbibliothek Hamburg Carl von Ossietzky}
}

@article{jacobs,
  title={Involutive categories and monoids, with a GNS-correspondence},
  author={Jacobs, Bart},
  journal={Foundations of Physics},
  volume={42},
  pages={874--895},
  year={2012},
  publisher={Springer}
}

@article{beninischenkelwoike,
  title={Involutive categories, colored *-operads and quantum field theory},
  author={Benini, Marco and Schenkel, Alexander and Woike, Lukas},
  journal={Theory and Applications of Categories},
  volume={34},
  number={2},
  pages={13--57},
  year={2019}
}

@article{higherdagger,
  title={Dagger $n$-categories},
  author={Ferrer, Giovanni and Hungar, Brett and Johnson-Freyd, Theo and Krulewski, Cameron and M{\"u}ller, Lukas and Penneys, David and Reutter, David and Scheimbauer, Claudia and Stehouwer, Luuk and Vuppulury, Chetan and others},
  journal={arXiv preprint arXiv:2403.01651},
  year={2024}
}

@article{baezdagger,
  title={Quantum quandaries: a category-theoretic perspective},
  author={Baez, John},
  journal={The structural foundations of quantum gravity},
  pages={240--265},
  year={2006},
  publisher={Oxford University Press}
}

@book{turaev,
  title={Monoidal categories and topological field theory},
  author={Turaev, Vladimir and Virelizier, Alexis},
  volume={322},
  year={2017},
  publisher={Springer}
}

@article{luukspinstatistics,
  title={The categorical spin-statistics theorem},
  author={Stehouwer, Luuk},
  journal={arXiv preprint arXiv:2403.02282},
  year={2024}
}

@article{durhuus1994classification,
  title={Classification and construction of unitary topological field theories in two dimensions},
  author={Durhuus, Bergfinnur and Jonsson, Thordur},
  journal={Journal of Mathematical Physics},
  volume={35},
  number={10},
  pages={5306--5313},
  year={1994},
  publisher={AIP Publishing}
}

@article{abramstqft,
  title={Two-dimensional topological quantum field theories and Frobenius algebras},
  author={Abrams, Lowell},
  journal={Journal of Knot theory and its ramifications},
  volume={5},
  number={05},
  pages={569--587},
  year={1996},
  publisher={World Scientific}
}

@article{sawin1995direct,
  title={Direct sum decompositions and indecomposable TQFTs},
  author={Sawin, Stephen},
  journal={Journal of Mathematical Physics},
  volume={36},
  number={12},
  pages={6673--6680},
  year={1995},
  publisher={American Institute of Physics}
}

@phdthesis{luukthesis,
  title={Unitary fermionic topological field theory},
  author={Stehouwer, Luuk},
  year={2024},
  school={Universit{\"a}ts-und Landesbibliothek Bonn}
}

@article{Cockett,
  title={Coherence of the double involution on *-autonomous categories},
  author={Cockett, JRB and Hasegawa, Masahito and Seely, RAG},
  journal={Theory and Applications of Categories},
  volume={17},
  number={2},
  pages={17--29},
  year={2006}
}

@book{kock2dtqft,
  title={Frobenius algebras and 2-d topological quantum field theories},
  author={Kock, Joachim},
  number={59},
  year={2004},
  publisher={Cambridge University Press}
}

@article{zhu2dhermitiantqft,
  title={The Hermitian axiom on two-dimensional topological quantum field theories},
  author={Zhu, Honglin},
  journal={Journal of Mathematical Physics},
  volume={64},
  number={2},
  year={2023},
  publisher={AIP Publishing}
}

@article{heunen2019limits,
  title={Limits in dagger categories},
  author={Heunen, Chris and Karvonen, Martti},
  journal={Theory \& Applications of Categories},
  volume={34},
  year={2019}
}

@phdthesis{wayofthedagger,
    author = {Karvonen, Martti},
    title = {The way of the dagger},
    school = {University of Edinburgh},
    year = {2019} 
}

@inproceedings{kellystreet,
  title={Review of the elements of 2-categories},
  author={Kelly, Max and Street, Ross},
  booktitle={Category Seminar: Proceedings Sydney Category Theory Seminar 1972/1973},
  pages={75--103},
  year={2006},
  organization={Springer}
}

@article{mullerstehouwer,
  title={Reflection Structures and Spin Statistics in Low Dimensions},
  author={M{\"u}ller, Lukas and Stehouwer, Luuk},
  journal={arXiv preprint arXiv:2301.06664},
  year={2023}
}

@article{srinivasan2021dagger,
  title={Dagger linear logic for categorical quantum mechanics},
  author={Comfort, Cole and Cockett, Robin and Srinivasan, Priyaa},
  journal={Logical Methods in Computer Science},
  volume={17},
  year={2021},
  publisher={Episciences. org}
}

@book{yau2020involutive,
  title={Involutive Category Theory},
  author={Yau, Donald},
  year={2020},
  publisher={Springer}
}

@misc{MO-evil,
    title = {Are dagger categories truly evil?},
    author = {Henriques, Andre and LeFanu Lumsdaine, Peter and others},
    howpublished = {\url{https://mathoverflow.net/questions/220032/are-dagger-categories-truly-evil}},
    year = {2015}
}

\end{document}